\documentclass{amsart}
\usepackage{amssymb}
\usepackage{amsfonts}
\usepackage{url}
\newcommand\al{\alpha}
\newcommand\Aut{\mathrm{Aut}}

\newtheorem{theorem}{Theorem}
\newtheorem{corollary}{Corollary}
\newtheorem{example}{Example}

\begin{document}
\title{Enumerating AG-monoids algebraically}
\author{Muhammad Shah}
\email{mshahmaths@gmail.com}
\address{DEPARTMENT OF MATHEMATICS, QUAID-I-AZAM UNIVERSITY, ISLAMABAD,
PAKISTAN}
\author{S. Shpectorov}
\email{sergeys@for.mat.bham.ac.uk}
\address{SCHOOL OF MATHEMATICS, UNIVERSITY OF BIRMINGHAM, UK.}
\email{s.shpectorov@bham.ac.uk}
\keywords{AG-monoids, monoids, algebraically, GAP, order 8.}
\begin{abstract}
An AG-monoid is an AG-groupoid (a groupoid satisfying the identity called left invertive law $(xy)z=(zy)x$) and having a left identiy. In this paper we enumerate AG-monoids algebraically and then implement them in GAP to compute them computationally.
\end{abstract}
\subjclass{}
\maketitle

An AG-monoid is an AG-groupoid (a groupoid satisfying the identity called left invertive law $(xy)z=(zy)x$) and having a left identity. AG-groupoids are extinsively studied for example see  \cite{Enumeration},\cite{Smarandache},\cite{classes}. On the other hand monoids are associative semigroups. we are finding AG-monoids from commutative monoids and then implement them in GAP to compute them computationally.

First we provide the rigorous foundation for the counting of AG-monoids.
It is based on the following result.

\begin{theorem} \label{construction}
Suppose $(S,+)$ is a commutative monoid and suppose $\al\in\Aut(S)$ satisfies $\al^2=1$.
Let the product be defined on $S$ by $a\cdot b:=\al(a)+b$. Then $(S,\cdot)$ is an AG-monoid.
Furthermore, every AG-monoid can be obtained in this way from a unique pair $(S,\al)$.
\end{theorem}

\begin{proof}
First, suppose that $(S,+)$ is a commutative monoid and $\al\in\Aut(S)$ with $\al^2=1$.
Let us check that the new product $\cdot$ satisfies the left invertive law.

Let $x,y,z\in S$. Then $(x\cdot y)\cdot z=\al(\al(x)+y)+z=\al^2(x)+\al(y)+z=x+\al(y)+z$, since $\al^2=1$.
Since also $(z\cdot y)\cdot x=z+\al(y)+x=x+\al(y)+z$ by commutativity, we conclude that $(x\cdot y)\cdot z=
(z\cdot y)\cdot x$. In order to show that $(S,\cdot)$ is an AG-monoid it remains to check that $0$ is a left
identity. Taking $x\in S$, we get $0\cdot x=\al(0)+x=0+x=x$, and the first claim is proven.

For the second claim we need to show that, given an arbitrary AG-monoid $(S,\cdot)$ (where the left
identity is denoted by $0$), we can
recover from it a suitable commutative monoid $(S,+)$ and an involutive automorphism $\al$.

Define $\al:S\to S$ by $\al(x)=x\cdot e$. Furthermore, define addition on $S$ by
$x+y:=\al(x)\cdot y$. We need to see that $(S,+)$ is a commutative monoid and that
$\al$ is an automorphism of it, as above.

We start by showing that $+$ satisfies commutativity. Using left invertive law, we have
$x+y=\al(x)\cdot y=(x\cdot 0)\cdot y=(y\cdot 0)\cdot x=\al(y)\cdot x=y+x$. For associativity, recall
that AG-monoids (being an AG-groupoid) satisfy the medial law $(a\cdot b)\cdot(c\cdot d)=
(a\cdot c)\cdot(b\cdot d)$, which implies, together with the left identity property, that
$a\cdot(b\cdot c)=b\cdot(a\cdot c)$ for all $a,b,c\in S$.

Now, using left invertive law and $a\cdot (b\cdot c)=b\cdot (a\cdot c)$, we get
$(x+y)+z=((x\cdot 0)\cdot y)+z=(((x\cdot 0)\cdot y)\cdot 0)\cdot z=(z\cdot 0)\cdot((x\cdot 0)\cdot y)
=(x\cdot 0)\cdot((z\cdot 0)\cdot y)=(x\cdot 0)\cdot((y\cdot 0)\cdot z)=x+((y\cdot 0)\cdot z)=
x+(y+z)$. Hence $(S,+)$ is associative. It remains to see that $0$ is the identity of $(S,+)$.
Indeed, $0+x=(0\cdot 0)\cdot x=0\cdot x=x$, since $0$ is the left identity of $(S,\cdot)$.
Thus, by commutativity of addition we have $x+0=0+x=x$. We have
shown that $(S,+)$ is a commutative monoid.

Turning to $\al$, notice that $\al^2(x)=(x\cdot 0)\cdot 0=(0\cdot 0)\cdot x=0\cdot x=x$ for every $x\in S$.
Therefore, $\al^2=1$, which in particular means that $\al$ is bijective. Also,
$\al(x+y)=(x+y)\cdot 0=((x\cdot 0)\cdot y)\cdot 0$. By the left invertive law, the latter is equal to
$(0\cdot y)\cdot(x\cdot 0)=y\cdot(x\cdot 0)$. On the other hand, $\al(x)+\al(y)=
((x\cdot 0)\cdot 0)\cdot(y\cdot 0)=((y\cdot 0)\cdot 0)\cdot(x\cdot 0)=((0\cdot 0)\cdot y)\cdot(x\cdot 0)=
(0\cdot y)\cdot(x\cdot 0)=y\cdot(x\cdot 0)$. Thus, $\al(x+y)=\al(x+\al(y)$, which shows that $\al$ is
an involutive automorphism of $(S,+)$.

Finally, $x\cdot y=\al^2(x)\cdot y=\al(x)+y$. This shows that $(S,\cdot)$ can be recovered from
$(S,+)$ in the prescribed way. Clearly, both $(S,+)$ and $\al$ were recovered from $(S,\cdot)$ in a
canonical way, which means that this pair is unique for $(S,\cdot)$.
\end{proof}

We now illustrate our construction with an example.

\begin{example}
We start with the following commutative monoid $S$:

\begin{center}
\begin{tabular}{l|llllll}
$+$ & $0$ & $1$ & $2$ & $3$ & $4$ & $5$ \\ \hline
$0$ & $0$ & $1$ & $2$ & $3$ & $4$ & $5$ \\
$1$ & $1$ & $5$ & $2$ & $3$ & $4$ & $0$ \\
$2$ & $2$ & $2$ & $2$ & $3$ & $3$ & $2$ \\
$3$ & $3$ & $3$ & $3$ & $3$ & $3$ & $3$ \\
$4$ & $4$ & $4$ & $3$ & $3$ & $4$ & $4$ \\
$5$ & $5$ & $0$ & $2$ & $3$ & $4$ & $1$%
\end{tabular}%
\end{center}
\end{example}

It can be checked that the permutation $\al=(1,5)(2,4)$ is an automorphism of this commutative monoid,
and it clearly has order two. Applying this $\al$ to $(S,+)$ we get the following non-associative
AG-monoid.

\begin{example}
\end{example}
\begin{center}
\begin{tabular}{l|llllll}%
$+$ & $0$ & $1$ & $2$ & $3$ & $4$ & $5$ \\ \hline
$0$ & $0$ & $1$ & $2$ & $3$ & $4$ & $5$ \\
$1$ & $5$ & $0$ & $2$ & $3$ & $4$ & $1$ \\
$2$ & $4$ & $4$ & $3$ & $3$ & $4$ & $4$ \\
$3$ & $3$ & $3$ & $3$ & $3$ & $3$ & $3$ \\
$4$ & $2$ & $2$ & $2$ & $3$ & $3$ & $2$ \\
$5$ & $1$ & $5$ & $2$ & $3$ & $4$ & $0$%
\end{tabular}%
\end{center}

In order to be able to count the number of non-isomorphic AG-monoids we need to make the uniqueness claim
in Theorem \ref{construction} even more precise, as follows.

\begin{theorem}
Suppose $(S,+)$ and $(S',+)$ are two commutative monoids and let $\al\in\Aut(S)$, $\al'\in\Aut(S')$
be their involutive automorphisms. Finally, let $(S,\cdot)$ (respectively, $(S',\cdot)$) be the AG-monoid
derived from $S$ and $\al$ (respectively, $S'$ and $\al'$). Then a mapping $\phi:S\to S'$ is an isomorphism
of $(S,\cdot)$ onto $(S',\cdot)$ if and only if $\phi$ is an isomorphism of $(S,+)$ onto $(S',+)$ and,
furthermore, $\phi\al=\al'\phi$.
\end{theorem}

\begin{proof}
We first assume that $\phi$ is an isomorphism from $(S,\cdot)$ onto $(S',\cdot)$. Then, for $x,y\in S$, we
have $\phi(x)+\phi(y)=(\phi(x)\cdot 0')\cdot\phi(y)$, where $0'$ is the left identity in $S'$. Note that
the left invertive law implies that every AG-monoid has a unique left identity, which means that $0'=\phi(0)$,
where $0$ is the left identity of $S$. Therefore,
$(\phi(x)\cdot 0')\cdot\phi(y)=(\phi(x)\cdot\phi(0))\cdot\phi(y)=\phi((x\cdot 0)\cdot y)=\phi(x+y)$.
We have shown that $\phi$ is an isomorphism of $(S,+)$ onto $(S',+)$.

Also, $\phi(\al(x))=\phi(x\cdot 0)=\phi(x)\cdot\phi(0)=\phi(x)\cdot 0'=\al'(\phi(x))$. Since this is true
for all $x\in S$, we conclude that $\phi\al=\al'\phi$, as claimed.

Conversely, suppose that $\phi$ is an isomorphism of $(S,+)$ onto $(S',+)$ satisfying $\phi\al=\al'\phi$.
Then, first of all, $\phi(0)=0'$, as these are the unique identity elements in the respective commutative
monoids. Thus, $\phi(x\cdot y)=\phi(\al(x)+y)=\phi(\al(x))+\phi(y)=\al'(\phi(x))+\phi(y)=\phi(x)\cdot\phi(y)$,
proving that $\phi$ is an isomorphism of $(S,\cdot)$ onto $(S',\cdot)$.
\end{proof}

As a consequence, we immediately get the following.

\begin{corollary}
Suppose $(S,+)$ is a commutative monoid and $\al,\al'\in\Aut(S)$ satisfy $\al^2=1=(\al')^2$. Then the
AG-monoids obtained from $S$ and $\al$, and from $S$ and $\al'$ are isomorphic if and only if
$\al$ and $\al'$ are conjugate in $\Aut(S)$.
\end{corollary}

We omit the proof. It is easy to see that an AG-group is associative if and only if it is commutative and if and
only if the left identity in it is a two sided identity. A further equivalent condition is that $\al=1$.
In particular, the number of non-associative AG-monoids obtainable from a particular commutative monoid $S$
is equal to the number of conjugacy classes of involutive (nonidentity) automorphisms in $\Aut(S)$.

The number of nonisomorphic commutative monoids is known up to order 10 \cite{Order10}. The list up to order 8 is available
in the SMALLSEMI package \cite{DM25} of GAP\cite{GAP}. Using functions from the AGGROUPOIDS package \cite{AGGROUPOIDS} of GAP\cite{GAP} we were able to apply the
above method to all those commutative monoids which resulted in the following table.

\begin{center}
{\tiny
\begin{table}[h] \label{tab1}
\begin{tabular}{|l|l|l|l|l|l|l|l|l|l|l|l|l|l|l|l|l|l|l|l|l|}
\hline
Order & $3$ & $4$ & $5$ & $6$ & $7$ & $8$ \\ \hline
Commutative monoids & $5$ & $19$ & $78$ & $421$ & $2637$ & $20486$ \\ \hline
Non-associative AG-monoids & $1$ & $6$ & $29$ & $188$ & $1359$ & 11386 \\ \hline
Total & $6$ & $25$ & $107$ & $609$ & $3996$ & $31872$ \\ \hline
\end{tabular}
\vskip .5cm
\caption{Number of AG-monoids of order $n$, $3\le n\le 8$}
\end{table}
}
\end{center}


\end{document}